\documentclass{amsart}

\usepackage[dvips]{color}
\usepackage{amsfonts}
\usepackage{multicol}
\usepackage{graphicx}
\usepackage{amsmath}
\usepackage{amssymb}
\usepackage{amsthm}
\usepackage{mathrsfs}
\newtheorem{theorem}{Theorem}[section]

\newtheorem{proposition}[theorem]{Proposition}

\theoremstyle{definition}

\theoremstyle{remark}
\newtheorem{remark}[theorem]{Remark}

\numberwithin{equation}{section}

\newfam\msbfam
\font\tenmsb=msbm10  \textfont\msbfam=\tenmsb
\font\sevenmsb=msbm7  \scriptfont\msbfam=\sevenmsb
\font\fivemsb=msbm5    \scriptscriptfont\msbfam=\fivemsb
\def\Bbb{\fam\msbfam \tenmsb}

\def\CC{{\Bbb C}}

\newfam\msbbfam
\font\tenmsbb=msbm10  scaled \magstep1 \textfont\msbbfam=\tenmsbb
\font\sevenmsbb=msbm7  scaled \magstep1 \scriptfont\msbbfam=\sevenmsbb
\font\fivemsbb=msbm5    scaled \magstep1 \scriptscriptfont\msbbfam=\fivemsbb

\begin{document}

\title{Lindel\"{o}f principle for domains in $\CC^2$ of finite type}

\author{Baili Min}
\address{Department of Mathematics, Washington University in St.\,Louis, Saint Louis, MO 63130}
\email{minbaili@math.wustl.edu}



\date{\today}


\begin{abstract} 
Recall the Lindel\"{o}f principle for the unit disc in $\CC$.
In this paper we will show some results about the Lindel\"{o}f principle with admissible convergence for domains in $\CC^2$ of finite type.
\end{abstract}

\maketitle

\section{Background}
Recall the classical Lindel\"{o}f principle:
\begin{theorem}
Let $f$ be a bounded holomorphic function on the unit disc $D \subseteq \mathbb{C}$. Suppose that the radial limit
\begin{displaymath}
\lim_{r \to 1-}f(r e^{i\theta}) \equiv \lambda \in \mathbb{C}
\end{displaymath}
of $f$ exists at the boundary point $e^{i\theta}$. Then $f$ has nontangential limit $\lambda$ at $e^{i \theta}$.
\end{theorem}

For several complex variables, we would like to find an analogous theorem. But what kind of convergence we should have for domains in $\CC^n, n\geqslant 2$? Strongly pseudoconvex domains have been well studied, can we generalize results for the case of finite type?

Here we must mention the notion of finite type. The type of a boundary point describes the order of contact at this point. If a domain is strictly pseudoconvex then it is also of finite type, but the converse is false. For the history of the notion of finite type and the modern techniques, please refer to \cite{BG}, \cite{C}, \cite{D1}, \cite{D2} and \cite{Kohn}.

In the case of a single complex variable, the appropriate approach region is the non-tangential one, while in the case of several complex variables, we are more interested in the \emph{admissible} approach regions, because for these regions, in both strongly pseudoconvex and finite type cases, the Fatou theorem works( see \cite{K}, \cite{Stein}, \cite{Lempert}, \cite{NSW} and \cite{Neff}), and they are sharp( see \cite{HS} and \cite{Min}). So we wonder: is there a Lindel\"{o}f principle for domains in $\CC^n, n \geqslant 2$ with  admisible convergence?  However, it turns out to be false as we can construct a counter example( see \cite{CK}). 

But it is true for another kind of convergence, that is, \emph{hypoadmissible} convergence:
\begin{theorem}
Let $B \subseteq \mathbb{C}^n$ be the unit ball. Let $f$ be a bounded holomorphic function and fix $P \in \partial B$. If the limit
\begin{displaymath}
\lim_{r \to 1-} f(rP) \equiv \lambda \in \mathbb{C}
\end{displaymath}
exists, then for any sequence $\{z^{(j)}\}_{j=1}^{\infty} \subseteq B$ that approaches $P$ hypoadmissibly, we have
\begin{displaymath}
\lim_{j \to \infty} f(z^{(j)}) = \lambda.
\end{displaymath}
\end{theorem}
More details, including the definition of hypoadmissible convergence, can be found in \cite{Chirka}, \cite{CK} and \cite{SK1}

We still wonder how we can have the \emph{admissible} limit, because it has been shown that admissible approach regions are optimal to some extent, for domains strongly pseudoconvex, and even further of finite type, and it is strictly stronger than hypoadmissible convergence.

In this paper I would like to provide two main results with admissible convergence, dealing with domains in $\CC^2$ of finite type. They are Theorem \eqref{main1} and Theorem \eqref{main2}, both based on the study of the shape of the admissible approach regions and the work on strongly pseudoconvex domains.

Throughout the text, let $\Omega$ be a domain of finite type in $\mathbb{C}^2$, characterized by a defining function $\rho$, with $m \geqslant 2$ being the maximal type. Suppose that $(1,0)$ is on the boundary, and its normal direction is $\langle1, 0\rangle$. We also suppose that $\frac{\partial \rho}{\partial z_1}$ and $\frac{\partial \rho}{\partial \overline{z}_1}$ do not vanish at $(1,0)$.

\section{$T$-approach}
There are some interesting results by Krantz in \cite{SK2}, which give admissible convergence. But in that paper, the work is done for strongly pseudoconvex domains. Here we can generalize them for domains of finte type in $\CC^2$.
\subsection{Shape of the admissible approach region}
Suppose $\omega^0=(\omega_1^0, \omega_2^0) \in \partial \Omega$. Then, in a small neighborhood $V=V_{\omega^0}$ of $\omega^0$, the complex holomorphic tangential vector field has a basis $L+i\overline{L}$, where $L=-\frac{\partial \rho}{\partial z_2}\frac{\partial}{\partial z_1}+\frac{\partial \rho}{\partial z_1}\frac{\partial}{\partial z_2}$ and $\overline{L}=-\frac{\partial \rho}{\partial \overline{z}_2}\frac{\partial}{\partial \overline{z}_1}+\frac{\partial \rho}{\partial \overline{z}_1}\frac{\partial}{\partial \overline{z}_2}$.

Then we can find a transverse vector field $T$ such that $L, \overline{L}$ and $T$ span the tangent space to $\partial \Omega$ at any point in $V$. Therefore, if $\mathscr{L}_{k-1}$ is an iterated commutator of degree $k-1$, we have $\mathscr{L}_{k-1} \equiv \lambda_{k-1}T \ \text{mod}(L, \overline{L})$. Define $\mathscr{L}_k=[L,\mathscr{L}_{k-1}]$. Let $\mathscr{M}_k$ be the collection of all these linearly independent iterated commutators with degree less or equal to $k$. Suppose $\mathscr{L} \in \mathscr{M}_k$, and that $\lambda_{\mathscr{L}}$ is the coefficient function of $T$ in the sense that $\mathscr{L} \equiv \lambda_{\mathscr{L}}T \ \text{mod}(L, \overline{L})$. Then we can define $\Lambda_k(z)$ and $D(z)$by:
\begin{displaymath}
\Lambda_k(z)=\sqrt{\sum_{\mathscr{L} \in \mathscr{M}_k}{\lambda_{\mathscr{L}}^2(z)}}, \end{displaymath} 
\begin{displaymath}
D(z)=\inf_{2 \leqslant k \leqslant \tau}\Big(\frac{\delta(z)}{\Lambda_k\big(\pi(z)\big)}\Big)^{1/k},
\end{displaymath}where $\pi$ is the Euclidean normal projection and $\delta(z)=|z-\pi(z)|$ is the ordinary distance of $z$ to the boundary.

A polarization $R(z, w)$ of $\rho$ is a $C^\infty$ complex-valued function satisfying that $R(z,z)=\rho(z)$, $\overline{\partial}_zR(z,w)$ vanishes to infinite order on $z=w$ and $R(z,w)-\overline{R(w,z)}$ vanishes to infinite order on $z=w$.

The admissible approach region  at $(1,0)$ is then defined by
\begin{equation}
\left\{
\begin{array}{lr}
|\pi(z)-(1,0)|< \alpha D(z),&\\
\Big|R\big(\pi(z),(1,0)\big)\Big|<\Lambda^{\alpha D(z)}(1,0).& \\
\end{array} \right.
\end{equation}

We want to estimate $R\big(\pi(z),(1,0)\big)$. First let $\pi(z)=(w_1, w_2)$.

By definition, $R\big((w_1, w_2),(1,0)\big)$ is a holomorphic function. We can expand it into a formal series around the point $(1,0)$. Ignoring higher terms, we know that $\mathscr{A}_1(1,0)$ is comparable to the region defined by
\begin{equation}
\left\{
\begin{array}{lr}
|(w_1, w_2))-(1,0)|<  D(z),&\\
\big|c_1(w_1-1)^{k_1}+c_2(w_1-1)^{k_2}w_2^{k_3}+c_1 w_2^{k_4}\big|<\Lambda^{ D(z)}(1,0),& \\
\end{array} \right.
\end{equation}
where $c_i$'s are complex numbers and $k_i$'s are positive integers.

As we estimate $D(z)$ and $\Lambda^{\alpha D(z)}(1,0)$( see \cite{Min}), we know that this region sits inside, up to comparability, the one defined by
\begin{equation}
\left\{
\begin{array}{lr}
|(w_1, w_2))-(1,0)|< \delta(z)^{\frac{1}{m}},&\\
\big|c_1(w_1-1)^{k_1}+c_2(w_1-1)^{k_2}w_2^{k_3}+c_1 w_2^{k_4}\big|<\delta(z)^k,& \\
\end{array} \right.
\end{equation}
where $k>0$.

By solving the inequalities, we know that it is comparable to the region given by
\begin{equation}
\left\{
\begin{array}{lr}
|(w_1, w_2)-(1,0)|< \delta(z)^{\frac{1}{m}},&\\
\big|w_1-1|<\delta(z)^{k_5},& \\
\end{array} \right.
\end{equation}
which is also comparable to, by doing the same analysis as in \cite{Min},
\begin{equation}
\left\{
\begin{array}{lr}
|(z_1, z_2)-(1,0)|< \delta(z)^{\frac{1}{m}},&\\
\big|w_1-1|<\delta(z)^{k_5}.& \\
\end{array} \right.
\label{bigregion}
\end{equation}

We have
\begin{equation}
|z_1-1| \leqslant |z_1-w_1|+|w_1-1|
\end{equation}
and
\begin{equation}
|z_1-w_1| \sim \delta(z), |w_1-1|<\delta(z)^{k_5}.
\end{equation}

CASE ONE: $k_5 \geqslant 1$.

In this case we have $ |z_1-1| \lesssim \delta(z)$. So up to some comparability, the region \eqref{bigregion} is inside 
\begin{equation}
\left\{
\begin{array}{lr}
|(z_1, z_2)-(1,0)|< \delta(z)^{\frac{1}{m}},&\\
|z_1-1|<\delta(z),& \\
\end{array} \right.
\label{bigregionvar1}
\end{equation}

CASE TWO: $k_5 < 1$.

In this case we have $ |z_1-1| \lesssim \delta(z)^{k_5}$. Therefore we have
\begin{equation}
\left\{
\begin{array}{lr}
|(z_1, z_2)-(1,0)|< \big(\delta(z)^{k_5}\big)^{\frac{1}{k_5m}}<\big(\delta(z)^{k_5}\big)^{\frac{1}{m}},&\\
|z_1-1|<\delta(z)^{k_5}.& \\
\end{array} \right.
\label{bigregionvar2}
\end{equation}

Therefore we can conclude that:
\begin{proposition}
\label{regionshape}
There is a region $\mathscr{A}$ that is comparable with $\mathscr{A}_1(1,0)$ and is lying inside the region bounded by 
\begin{displaymath}
|(z_1, z_2)-(1,0)|^m=|z_1-1|.
\end{displaymath}
\end{proposition}

Then following the Krantz' arguments in \cite{SK2}, we can get some similar results, to be discussed in the following subsections.

\subsection{Bounded holomorphic functions}

First we define a two-dimensional and totally real region:
\begin{displaymath}
T=\{(s+i0, t+i0)\in \Omega: s, t \in \mathbb{R}, 0<s<1, 0<|t|<\sqrt[m]{1-s}\}.
\end{displaymath}

Then for $j=1, 2, \ldots$, we define
\begin{displaymath}
\Omega_j=\{(z_1, z_2) \in \Omega: 1-2^{-j} \leqslant \Re z_1 < 1-2^{-j-1}, |\Im z_1| < 2^{-j} \text{ and }|z_2|< \sqrt[m]{2^{-j}}\}.
\end{displaymath}

For each $\Omega_j$, the map 
\begin{displaymath}
\varphi_j(z_1, z_2)=(2^{j-j_0}(z_1-1)+1, \sqrt[m]{2^{j-j_0}}z_2)
\end{displaymath}
gives a biholomorphic mapping from $\Omega_j$ onto a region $\Omega_{j_0}$, where $j_0$ is a positive integer.

By Proposition \eqref{regionshape}, we have
\begin{equation}
\mathscr{A} \subseteq \bigcup_{j=1}^{\infty} \Omega_j =\bigcup_{j=1}^{\infty} \varphi_j^{-1}(\Omega_0) .
\end{equation}

\begin{proposition}
Let $f$ be a bounded holomorphic function in $\Omega$. If 
\begin{displaymath}
\lim_{T \ni z \to (1,0)} f(z) = 0,
\end{displaymath} then
\begin{displaymath}
\lim_{\mathscr{A} \ni z \to (1,0)} f(z) = 0.
\end{displaymath}
\end{proposition}

\begin{proof}
First we see that $\varphi_j$ maps $T \cap \Omega_j$ onto $T \cap \Omega_{j_0}$.

For each $j$, construct
\begin{displaymath}
g_j=f \circ \varphi_j^{-1}: \Omega_{j_0} \to \mathbb{C}.
\end{displaymath}
They are uniformly bounded, so they form a normal family, and therefore we can find $g_0$, a subsequential limit function.

Note that $g_0$ vanishes on $T \cap \Omega_{j_0}$, a totally real two-dimensional region. It follows that $g_0$ vanishes identically.

For any compact set $K \subseteq \Omega_{j_0}$ such that 
\begin{displaymath}
\mathscr{A} \subseteq \bigcup_{j=1}^{\infty}\varphi_j^{-1}(K),
\end{displaymath}
we know that $g_j \to 0$ uniformly on $K$. Therefore $f$ has $\mathscr{A}$-admissible limit 0. 
\end{proof}
\begin{remark}
As seen in the analysis above, the crucial part is $\mathscr{A} \subseteq \bigcup_{j=1}^{\infty} \Omega_j =\bigcup_{j=1}^{\infty} \varphi_j^{-1}(\Omega_0)$. If we have more information about types of points near $(1,0)$, we may make $T$ sharper. 
\end{remark}

Now define a two-dimensional, totally real manifold
\begin{displaymath}
\mathscr{T}=\{\big(s+i\rho_1(s,t), t+i\rho_2(s,t)\big):(s, t) \in T\},
\end{displaymath}
where $\rho_i(s,t): T \to \mathbb{R}$ is a $C^2$ function, $i=1,2$.

Then we have
\begin{proposition}
Let $f$ be a bounded holomorphic function in $\Omega$. If 
\begin{displaymath}
\lim_{\mathscr{T} \ni z \to (1,0)} f(z) = 0,
\end{displaymath}then
\begin{displaymath}
\lim_{\mathscr{A} \ni z \to (1,0)} f(z) = 0.
\end{displaymath}
\end{proposition}
\begin{proof}
Suppose 
\begin{equation}
\varphi_j(\mathscr{T} \cap \Omega_j)= \tau_j(T \cap \Omega_{j_0}),
\label{relation}
\end{equation}
then each $\tau_j$ has bounded derivatives. So we can find a subsequence $\{\tau_{j_k}\}$ that converges uniformly on compacta to $\tau_0$. By the relation \eqref{relation} and the definition of $g_j$, we know there exists a convergent subsequence $\{g_{j_k}\}$ with the limit $g_0$. 

Let $\mathscr{T}_0$ be the graph of $\tau_0$ over $T \cap \Omega_{j_0}$. We observe that $\mathscr{T}_0$ is a totally real, two-dimensional manifold.

We claim that $g_0$ vanishes on $\mathscr{T}_0$. It is true because for any $w \in \mathscr{T}_0$, there exists a point $z_{j_k} \in \mathscr{T} \cap \Omega_{j_k}$ such that $\varphi_{j_k}^{-1}(w)=z_{j_k}$. We then know that $\{z_{j_k}\}$ lies in $\mathscr{T}$ and approaches to $(1,0)$ as $k \to \infty$. By the hypothesis we know that 
\begin{equation}
\lim_{k \to \infty}f\circ \varphi_{j_k}^{-1}(w)=0,
\end{equation}
and therefore the claim is verified.

So we can again conclude that $g_0 \equiv 0$ and it then follows that $f$ has $\mathscr{A}$-limit 0 at $(1,0)$.
\end{proof}

\subsection{Normal functions}
First recall that holomorphic function $f: \Omega \to \hat{\mathbb{C}}$ is normal if the derivative $\nabla f$ is bounded from the Kobayashi metric on $\Omega$ to the spherical metric on $\hat{\mathbb{C}}$. This is a generalization of the normal functions of a single complex variable. For more details, pleace refer to \cite{CK}.
\begin{theorem}
\label{main1}
Let $f$ be a normal holomorphic function in $\Omega$. If 
\begin{displaymath}
\lim_{\mathscr{T} \ni z \to (1,0)} f(z) = 0,
\end{displaymath} then
\begin{displaymath}
\lim_{\mathscr{A} \ni z \to (1,0)} f(z) = 0.
\end{displaymath}
\end{theorem}
\begin{proof}
Let $D$ be the unit disc in $\mathbb{C}$.
Consider a holomorphic mapping $\psi : D \to \Omega_{j_0}$ with $\psi(0)=p \in \Omega_{j_0}$.  We may take it to be an extremal function for the Kobayashi metric at the point $p$.

Define a function $\mu_j : D \to \hat{\mathbb{C}}$:
\begin{displaymath}
\mu_j=f\circ \varphi_j^{-1} \circ \psi.
\end{displaymath}

It then follows
\begin{equation}
|\mu_j '(0)| \leqslant \big|\nabla f(\varphi_j^{-1}(p))\big|\big|\big(\varphi_j^{-1} \circ \psi\big)'(0)\big|.
\label{inequality}
\end{equation}

We notice that $\big|\nabla f(\varphi_j^{-1}(p))\big|$ is bounded from the Kobayashi metric on $\Omega$, and $\big|\big(\varphi_j^{-1} \circ \psi\big)'(0)\big|$ is the reciprocal of the Kobayashi metric for $\Omega_j$ at $\varphi_j^{-1}(p)$. We also notice that the Kobayashi metric on $\Omega$ is smaller than that on $\Omega_j$. Therefore we can see that $|\mu_j '(0)|$ is bounded on compact subset of $D$, and this bound is independent of $j$, and the choice of $p$ in a compact subset $K \subseteq D$. By composing a M\"{o}bius transformation we can have a similar estimate for $\mu_j'$ at any point of a compact subset of $D$.

Therefore we can find  a normally convergent subsequence $\{\mu_{j_k}'\}$ of $\mu_j '$ with the limit function $\mu_0'$.  Consequently, $\{\mu_{j_k}=g_{j_k}\circ\psi\}$ is also convergent, and so is $\{g_{j_k}\}$, with the limit function $g_0$. 

As shown in the proof of the previous theorem,  we can obtain a totally real, two-dimensional manifold $\mathscr{T}_0$, the graph of $\tau_0$, a subsequential limit of $\{\tau_j\}$. We can further deduce that $g_0$ vanishes on $\mathscr{T}_0$, then get that $g_0 \equiv 0$ and finally conclude that $f$ has $\mathscr{A}$-limit 0 at $(1,0)$.
\end{proof}

\section{Boundary approach}
Suppose $f$ is a bounded holomorphic function on $\Omega$ such that $|f(z)| \leq 1$ for any $z \in \Omega$. For $P \in \partial \Omega$,the boundary value $|f(P)|$ is defined to be $\limsup_{\Omega \ni z \to P}|f(z)| \in \mathbb{R}\cup\{\infty\}$. 

There are some interesting discoveries in \cite{LV} and \cite{CK} concerning the boundary curves. The result in \cite{LV} of Lehto and Virtanen states that if $\lim_{t \to 1-}|f(\gamma(t))|=0$ where $f$ is a normal function in $D \subseteq \CC$ and $\gamma : [0,1] \to \overline{D}$ is such a curve that $\gamma(1)=P \in \partial D$, then $f$ has non-tangential limit 0 at $P$. For several complex variables, Cima and Krantz gives a similar result for hypoadmissible convergence in \cite{CK}, which adopts complex normal curves. Recall that $\gamma : [0,1] \to \partial \Omega$ is complex normal if $\langle\gamma'(t),\nu_{\gamma(t)}\rangle\neq0$, all $0 \leqslant t \leqslant 1$. Then their theorem states:
\begin{theorem}
Let $\Omega \subset \subset \CC^n$ be a domain with $C^2$ boundary. Let $\gamma:[0,1] \to \partial \Omega$ be a $C^2$ curve which is complex normal. Let $f: \Omega \to \CC$ be a normal and assume $f \in H^p(\Omega)$, $p > 4n$. Suppose that
\begin{displaymath}
\lim_{t \to 1-}|f(\gamma(t))|=0.
\end{displaymath} 
Then $f$ has hypoadmissible limit 0 at $P \equiv \gamma(1)$.
\end{theorem}

However, there is no way to get an admissible limit. Let us consider this domain 
\begin{displaymath}
\Omega_2=\{(z_1,z_2) \in \mathbb{C}^2: |z_1|^2+|z_2|^4<1\},
\end{displaymath}
which is not strongly pseudoconvex but still of finite type.

Consider the bounded holomorphic function
\begin{displaymath}
f(z)=f(z_1, z_2)=\frac{z_2^4}{1-z_1}.
\end{displaymath}

We notice that $f$ has a radial limit 0. If there exists a complex normal curve terminating at $(1,0)$, along which $f$ has a limit $\lambda \neq 0$, then according to some other results in \cite{CK} by Cima and Krantz, $f$ should have a hypoadmissible limit $\lambda$ and thus radial limit $\lambda$, which gives a contradition. This means, along \emph{any} complex normal curve, if $f$ has a limit, then this limit must be 0. Or we can just check this curve $\varphi(t)=(e^{it},0), 0 \leq t \leq \frac{\pi}{2}$, and note that it is complex normal and along it $f$ has the limit 0.

However, this does not yield the admissible limit 0, as we can find a sequence of points $\{z^{(k)}\}_{k=0}^{\infty}$ with $z^{(k)}=(1-2^{-4k}, 2^{-k})$, which are in an admissible approach region, and
\begin{equation}
f(z^{(k)})=\frac{2^{-4k}}{2^{-4k}}=1.
\end{equation}
Therefore we have the limit
\begin{equation}
\lim_{k \to \infty}f(z^{(k)})=1.
\end{equation}

So we need to put stricker conditions.

Suppose 
\begin{displaymath}
\lim_{t \to 1^{-}}f(\varphi(t))=\ell
\end{displaymath}
for any boundary curve $\varphi: [0,1] \to \partial \Omega$  with $\varphi(1)=(1,0)$. We hope to find a Lindel\"{o}f principle for this case, that is, we wish that $f$ had the admissible limit $\ell$ at $(1,0)$.

Since the boundary values of $f$ are defined through nontangential limit and along all cuvers near $(1,0)$, $|f|$ is defined, we may assume that there exists a neighborhood $W \subset \partial \Omega$ of $(1,0)$ such that 
\begin{displaymath}
\lim_{W \ni \omega \to (1,0)}f(w) = \ell,
\end{displaymath}
and may even assume that $W$ is also part of the boundary of another domain $V$ inside $\Omega$ that is of $C^2$ boundary, and $f$ has the nontangential limit at every point in $W$.

As stated in the previous section, we may just consider the region $\mathscr{A}$ which is comparable with $\mathscr{A}_1(1,0)$. So we hope to get this result:
\begin{displaymath}
\lim_{\mathscr{A}\ni z \to (1,0)}|f(z)-f(1,0)|=0.
\end{displaymath}

So we begin to estimate $|f(z)-f(1,0)|$.

First we have the triangle inequality
\begin{align}
|f(z)-f(1,0)| &\leq |f(z)-f((1,0)-\frac{1}{k}\nu)|  \nonumber \\
              &+|f((1,0)-\frac{1}{k}\nu)-f(1,0)|  
\end{align}
for any $k \in \mathbb{N}$, where $\nu$ is the outward unit normal vector at $(1,0)$.

We have no worry about the second expression  because it has the limit 0 when $k \to \infty$, so we hope to have the limit 0 for the first expression. To see this, we use Poisson integral over $\partial V$. So it turns out to be the problem to estimate in terms of Poisson kernels. 

For any positive $\varepsilon$ small enough, define
\begin{displaymath}
W_{\varepsilon}=\{P \in W: |P-(1,0)|<\varepsilon^3\},
\end{displaymath}
then $\sigma(W_{\varepsilon}) \sim \varepsilon^9$.

\begin{align}
|f(z)-f((1,0)-\frac{1}{k}\nu)| &\leq\int_{\partial V}{|P(z, \zeta)-P((1,0)-\frac{1}{k}\nu, \zeta)||f(\zeta)|\,d\sigma(\zeta)} \nonumber\\
&=\int_{\partial V- W_{\varepsilon}}{|P(z, \zeta)-P((1,0)-\frac{1}{k}\nu, \zeta)||f(\zeta)|\,d\sigma(\zeta)} \cdots (*) \nonumber \\
&+\int_{W_{\varepsilon}}{|P(z, \zeta)-P((1,0)-\frac{1}{k}\nu, \zeta)||f(\zeta)|\,d\sigma(\zeta)} \cdots (**) 
\end{align}

We are not worried about (*) because on $\partial V- W_{\varepsilon}$, as $z$ is approaching to $(1, 0)$ and $k$ is tending to $\infty$, $\zeta$ is away from the singularities of the Poisson kernels, and $|f(\zeta)|$ is bounded. Therefore the expression (*) has the limit 0.

We know that $P(z, \zeta)$ is comparable to $\delta(z)/|z-\zeta|^4$, so we want to estimate, for $\zeta \in W_{\varepsilon}$,
\begin{equation}
|\frac{\delta(z)}{|z-\zeta|^4}-\frac{\delta((1,0)-\frac{1}{k}\nu)}{|(1,0)-\frac{1}{k}\nu-\zeta|^4}|.
\end{equation}

If $C_1 \varepsilon <\delta(z) < C_2 \varepsilon$ and $C_3 \varepsilon < \frac{1}{k} < C_4 \varepsilon$, we have 
\begin{equation}
|z-\zeta|^4 \geq (\delta(z))^4 > C_1^4 \varepsilon^4
\end{equation}
and 
\begin{equation}
|(1,0)-\frac{1}{k}\nu|^4 = (\frac{1}{k})^4 > C_3^4 \varepsilon^4.
\end{equation}

We also need their upper bounds. By the triangle inequalily we know that
\begin{equation}
|z-\zeta| \leq |z-(1,0)|+|\zeta-(1,0)|.
\end{equation}

Since $z \in \mathscr{A}$, as shown in \eqref{bigregion},there is the relation $|z-(1,0)|<\delta(z)^{\frac{1}{m}}$.

By the definition of $W_{\varepsilon}$, we have $|\zeta-(1,0)| < \varepsilon^3$. So we can estimate that
\begin{equation}
|z-\zeta| < C_6 \varepsilon^{\frac{1}{m}}.
\end{equation}

However, we estimate that
\begin{equation}
|(1,0)-\frac{1}{k}\nu-\zeta| \leq |(1,0)-\frac{1}{k}\nu -(1,0)| + |\zeta -(1,0)|< C_7 \varepsilon.
\end{equation}
Therefore we have 
\begin{equation}
\Big|\delta(z)|(1,0)-\frac{1}{k}\nu-\zeta|^4-\frac{1}{k}|z-\zeta|^4\Big| < C_8 \cdot \varepsilon \cdot (\varepsilon^{\frac{1}{m}})^4=C_8 \varepsilon^{1+\frac{4}{m}}.
\end{equation}

Now we can estimate that
\begin{align}
|\frac{\delta(z)}{|z-\zeta|^4}-\frac{\delta((1,0)-\frac{1}{k}\nu)}{|(1,0)-\frac{1}{k}\nu-\zeta|^4}|
&=
|\frac{\delta(z)}{|z-\zeta|^4}-\frac{\frac{1}{k}}{|(1,0)-\frac{1}{k}\nu-\zeta|^4}| \nonumber \\
&=\frac{\Big|\delta(z)|(1,0)-\frac{1}{k}\nu-\zeta|^4-\frac{1}{k}|z-\zeta|^4\Big|}{|z-\zeta|^4|(1,0)-\frac{1}{k}\nu-\zeta|^4} \nonumber \\
&\leq \frac{C_9}{\varepsilon^8}\big|\delta(z)|(1,0)-\frac{1}{k}\nu-\zeta|^4-\frac{1}{k}|z-\zeta|^4\big| \nonumber \\
&< \frac{C_9}{\varepsilon^8} \cdot C_8 \varepsilon^{1+\frac{4}{m}} \nonumber \\
&=\frac{C_{10}}{\varepsilon^{7-\frac{4}{m}}}. 
\end{align}

Therefore, by the boundedness of $f$ on the boundary, we see that
\begin{align}
\int_{W_{\varepsilon}}{|P(z, \zeta)-P((1,0)-\frac{1}{k}\nu, \zeta)||f(\zeta)|\,d\sigma(\zeta)}
&<C_{11}\varepsilon^9\cdot\frac{C_{10}}{\varepsilon^{7-\frac{4}{m}}} \nonumber \\
&=C_{12}\varepsilon^{2+\frac{4}{m}} \nonumber \\
&<C_{12}\varepsilon^2. 
\end{align}

This means $|f(z)-f((1,0)-\frac{1}{k}\nu)|$ has the limit 0 as $z$ approaches to $(1,0)$ admissibly and $k$ tends to infinity.

Therefore we can establish this result:
\begin{theorem}
\label{main2}
Let $\Omega$ be a sdomain in $\mathbb{C}^2$ that is of finite type.
Suppose $f$ is a bounded holomorphic function on $\Omega$ such that $|f(z)| \leq 1$ for any $z \in \Omega$, and 
\begin{displaymath}
\lim_{t \to 1^{-}}f(\varphi(t))=\ell
\end{displaymath}
for any boundary curve $\varphi: [0,1] \to \partial \Omega$  with $\varphi(1)=(1,0)$. Then
\begin{displaymath}
\lim_{\mathscr{A}\ni z \to (1,0)}|f(z)-f(1,0)|=0.
\end{displaymath}
\end{theorem}

\bibliographystyle{amsplain}

\begin{thebibliography}{17}
\bibitem {BG} T.\,Bloom and I.\,Graham \textit{A geometric characterization of points of type $m$ on real submanifolds of $\mathbb{C}^n$}, J.\, Diff.\,Geom. \textbf{12}(1977), 171-182.

\bibitem {C} D.\,Catlin \textit{Subelliptic estimates for the $\overline{\partial}-$Neumann problem}, Ann.\,Math. \textbf{117}(1983), 147-172.

\bibitem {Chirka} E.\,Chirka \textit{The theorems of Lindel\"{o}f and Fatou in $\CC^n$}, Mat.\,Sb. \textbf{92}(134)(1973), 622-644; Math.\,U.S.S.R. Sb. \textbf{21}(1973), 619-639.

\bibitem {CK} J.\,A.\,Cima, S.\,G.\,Krantz \textit{A Lindel\"{o}f principle and normal functions in several complex variables}, Duke Math.\,J.50(1983), 303-328.

\bibitem {D1} J.\,P.\,D'Angelo \textit{Finite type conditions for real hypersurfaces in $\mathbb{C}^n$}, in \textit{Complex Analysis Seminar}, Springer Lecture Notes vol.\,1268, Springer Verlag, 1987, 83-102.

\bibitem {D2} J.\,P.\,D'Angelo \textit{Iterated commutators and derivatives of the Levi form}, in \textit{Complex Analysis Seminar}, Springer Lecture Notes vol.\,1268, Springer Verlag, 1987, 103-110.

\bibitem {HS} M.\,Hakim and N.\,Sibony \textit{Fonctions holomorphes born\'{e}es et limites tangentielles},
Duke Math. J. \textbf{50}(1983), 133-141.

\bibitem {Kohn} J.\,J.\,Kohn \textit{Boundary behavior of $\overline{\partial}$ on weakly pseudoconvex manifolds of dimension two}, J.\, Diff.\,Geom. \textbf{6}(1972), 523-542.

\bibitem{K} A.\,Kor\'{a}nyi \textit{harmonic functions on Hermitian hyperbolic space}, Trans.\,Amer.\,Math.\,Soc. 135(1969), 507-516.

\bibitem {SK1} S.\,G.\,Krantz \textit{Function Theory of Several Complex Variables,  $2^{nd}$ ed.}, AMS Chelsea Publishing, RI, 2000.


\bibitem {SK2} S.\,G.\,Krantz \textit{The Lindel\"{o}f principle in several complex variables}, J.\,Math.\,Anal.\,Appl. 326(2007) 1190-1198.

\bibitem {LV} O.\,Lehto, K.\,I.\,Virtanen \textit{Boundary behavior and normal meromorphic functions}, Acta.\,Math. 97(1957) 47-65.

\bibitem {Lempert} L.\,Lempert \textit{Boundary behavior of meromorphic functions of several complex variables},
Acta Math. 144(1980), 1-26.

\bibitem{Min} B.\,Min \textit{Approach regions in domain in $\CC^2$ of finite type} preprint arXiv:1002.1524v2.

\bibitem {NSW} A.\,Nagel, E.\,M.\,Stein and S.\,Wainger, \textit{Boundary behavior of functions holomorphic in domains of finite type}, Proc. Nat. Acad. Sci. USA 78(1981), 6596-6599.

\bibitem {Neff} C.\,A.\,Neff \textit{Maximal Function Estimates for Meromorphic Nevanlinna Functions},
PhD Thesis, Princeton University, 1985.

\bibitem {Stein} E.\,M.\,Stein \textit{Boundary Behavior of Holomorphic Functions of Several Complex Variables},
Princeton University Press, Princeton, NJ, 1972.


\end{thebibliography}

\end{document}